\documentclass[12pt]{amsart}

\usepackage{geometry}
\usepackage{indentfirst}
\usepackage{enumitem}
\usepackage{amssymb}
\usepackage{mathtools}
\usepackage{stix2}
\usepackage{comment}
\usepackage[colorlinks,linkcolor=red,anchorcolor=green,citecolor=blue]{hyperref}
\hypersetup{linktocpage = true}

\makeatletter
\def\@evenhead{\null\hfill\thepage\hfill\null} 
\def\@oddhead{\null\hfill\thepage\hfill\null}  
\def\@evenfoot{\null}
\def\@oddfoot{\null}
\makeatother

\allowdisplaybreaks
\newcommand{\abs}[1]{\left|#1\right|}
\newcommand{\ton}[1]{\left(#1\right)} 
\newcommand{\ps}[2]{\langle{#1},{#2}\rangle}

 \newcommand{\Ric}{\mathrm{Ric}}

\theoremstyle{plain}
\newtheorem{thm}{Theorem}[section]
\newtheorem{lemma}{Lemma}[section]
\newtheorem{prop}{Proposition}[section]
\newtheorem{cor}{Corollary}[section]

\theoremstyle{definition}

\newtheorem{remark}{Remark}[section]

\setlist[itemize]{leftmargin=*}
\setlist[enumerate]{leftmargin=*}

\numberwithin{equation}{section}

	\title[
]{Sharp Estimates for the Robin Laplacian under a Perimeter Constraint in hyperbolic space
}

\author{Daguang Chen, Shan Li }
\thanks{The authors were supported by NSFC grant No. 11831005 and NSFC-FWO W2521103.}

\subjclass[2020]{{35P15}, {58J50}, {52A55}
}
\keywords{Robin Laplacian; isoperimetric inequalities; horospherically convex; boundary parameter}

\begin{document}

	\begin{abstract}	
	In this paper, we establish a quantitative lower bound, expressed in terms of
	the isoperimetric deficit, for the normalized first eigenvalue deficit of the
	Robin Laplacian with a negative boundary parameter on bounded horospherically
	convex domains in hyperbolic space. As a consequence, the geodesic ball
	maximizes the first Robin eigenvalue within this class, thereby resolving the higher-dimensional counterpart of an open problem posed  by Celentano, Krej\v{c}i\v{r}\'{i}k and Lotoreichik in \cite{CKL26}. Furthermore, for a
	positive boundary parameter, we derive a quantitative upper bound for the
	normalized difference between the first Robin eigenvalue of such a domain and
	that of the geodesic ball with the same perimeter. 
	\end{abstract}
	
		\maketitle
\section{Introduction}
Let $\Omega$ be a bounded domain with smooth boundary in an
$n$-dimensional complete Riemannian manifold $\ton{M,g}$. We consider the
following Robin eigenvalue problem:
\begin{equation}\label{eq:rob}
	\begin{cases}
		-\Delta u=\lambda(\Omega,\beta)u, & \text{in }\Omega,\\[1mm]
		\dfrac{\partial u}{\partial N}+\beta u=0, & \text{on }\partial\Omega.
	\end{cases}
\end{equation}
Here, $N$ denotes the outward unit normal vector field along $\partial\Omega$,
and $\beta\in\mathbb{R}$ is the Robin boundary parameter. The Rayleigh quotient
associated with \eqref{eq:rob} is
\begin{equation*}
	\mathcal{R}_{\beta}[u]
	=
	\frac{\displaystyle\int_{\Omega}|\nabla u|^2\,d\mu
		+\beta\displaystyle\int_{\partial\Omega}u^2\,d\sigma}
	{\displaystyle\int_{\Omega}u^2\,d\mu},
	\qquad u\in H^1(\Omega)\setminus\{0\}.
\end{equation*}
The first eigenvalue of the Robin Laplacian admits the variational
characterization
\begin{equation}\label{eq:ray}
	\lambda_1(\Omega,\beta)
	=\min_{u\in H^1(\Omega)\setminus\{0\}}\mathcal{R}_{\beta}[u].
\end{equation}

When $\beta<0$, Bareket \cite{Bar77} conjectured in 1977 that, among all
Lipschitz planar domains of prescribed area, the disk maximizes
$\lambda_1(\Omega,\beta)$. In 2015, Ferone, Nitsch, and Trombetti
\cite{FNT15} proved that the ball is a local maximizer of the first Robin
eigenvalue among bounded Lipschitz domains in $\mathbb{R}^n$. Freitas and
Krej\v{c}i\v{r}\'{i}k \cite{FK15} subsequently verified Bareket's conjecture
for negative values of $\beta$ sufficiently close to zero in the Euclidean
plane. On compact Riemannian manifolds, Savo \cite{Savo20} obtained an upper
bound for the first Robin eigenvalue depending on the dimension, the inradius,
and lower bounds for the Ricci curvature and the boundary mean curvature. Li
and Wang \cite{LW21} later established an upper bound for the first eigenvalue of
the Robin $p$-Laplacian in terms of the same geometric quantities by comparison
with an associated one-dimensional eigenvalue problem. Li, Wang, and Wu
\cite{LWW23} proved that, in nonpositively curved space forms, geodesic balls
maximize the second Robin eigenvalue among bounded domains of prescribed
volume.

For bounded $C^2$ planar domains $\Omega\subset\mathbb{R}^2$, Antunes,
Freitas, and Krej\v{c}i\v{r}\'{i}k \cite{AFK17} proved that
\begin{equation}\label{eq:per}
	\lambda_1(\Omega,\beta)\leq\lambda_1(\Omega^\star,\beta),
\end{equation}
where $\Omega^\star$ denotes the disk having the same perimeter as $\Omega$.
Bucur et al.\ \cite{BFN+19} established \eqref{eq:per} for convex domains in
$\mathbb{R}^n$. Amato, Gentile, and Masiello \cite{AGM24} obtained an
isoperimetric deficit lower bound for the first eigenvalue of the Robin
$p$-Laplacian on convex subsets of $\mathbb{R}^n$ with prescribed perimeter.
More eigenvalue estimates of this type can be found in
\cite{FL21,KL22,Vik22} and the references therein.

Acampora et al.\ \cite{ACC+25} established an isoperimetric-deficit lower
bound for the first Robin eigenvalue with a negative boundary parameter on
strongly convex domains in $\mathbb{S}^n$ under a perimeter constraint. They
also exhibited convex domains in $\mathbb{H}^n$ whose inner parallel sets fail
to remain convex. In the introduction to \cite{CKL26}, Celentano,
Krej\v{c}i\v{r}\'{i}k, and Lotoreichik posed the open problem of determining
whether, among all bounded smooth simply connected domains in the hyperbolic
plane with prescribed perimeter, the geodesic disk maximizes the lowest Robin
eigenvalue for a negative boundary parameter. We observe that the inner
parallel sets of a bounded horospherically convex domain remain
horospherically convex in $\mathbb{H}^n$. Using this observation, for
$\beta<0$, we derive a quantitative lower bound, in terms of the isoperimetric
deficit, for the normalized eigenvalue gap between a bounded horospherically
convex domain and the geodesic ball with the same perimeter. Consequently, the
geodesic ball maximizes the first Robin eigenvalue within this class, thereby resolving the higher-dimensional counterpart of an open problem posed by Celentano, Krej\v{c}i\v{r}\'{i}k and Lotoreichik in \cite{CKL26}.
\begin{thm}\label{thm1}
	Let $\Omega$ be a bounded horospherically convex domain with smooth boundary in hyperbolic space $\mathbb{H}^n$ ($n\geq 3$)
and $P(\Omega)$ be the perimeter of $\Omega$. Let $\Omega^\star$ be the geodesic ball with the same perimeter as $\Omega$, that is, $P\ton{\Omega}=P\ton{\Omega^\star}$. For $\beta<0$, denote by $v$ the eigenfunction associated with $\lambda_1(\Omega^\star,\beta)$ for \eqref{eq:rob}. Then
	\begin{equation}\label{eq:low}
		\frac{\lambda_1(\Omega^\star,\beta)-\lambda_1(\Omega,\beta)}
		{|\lambda_1(\Omega,\beta)|}
		\geq
		\frac{v_m^2}{\|v\|_{L^2(\Omega^\star)}^2}
		\bigl(|\Omega^\star|-|\Omega|\bigr),
	\end{equation}
	where
	\[
	v_m:=\min_{x\in\Omega^\star}v(x),
	\]
	and $|\cdot|$ denotes hyperbolic volume.
\end{thm}

Combining Theorem \ref{thm1} with the hyperbolic isoperimetric inequality
\cite{Cha84} yields the following consequence.

\begin{cor}\label{cor1}
	Under the assumptions of Theorem \ref{thm1},
	\begin{equation}\label{eq:cor1}
		\lambda_1(\Omega,\beta)\leq\lambda_1(\Omega^\star,\beta).
	\end{equation}
	Moreover, equality holds if and only if $\Omega$ is a geodesic ball, up to an
	isometry of $\mathbb{H}^n$.
\end{cor}

\begin{remark}
	Bucur et al.\ \cite{BFN+19} established the Euclidean counterpart of
	\eqref{eq:cor1} for convex domains in $\mathbb{R}^n$. Acampora et al.\
	\cite{ACC+25} proved the spherical counterparts of both \eqref{eq:low} and
	\eqref{eq:cor1} for strongly convex domains in $\mathbb{S}^n$.
\end{remark}

For $\beta>0$, Bossel \cite{Bos86} and Daners \cite{Dan06} proved that, in Euclidean space,
\begin{equation}\label{ineq: BD}
	\lambda_1(\Omega,\beta)\geq\lambda_1(\Omega^\sharp,\beta),
\end{equation}
where $\Omega^\sharp$ denotes the Euclidean ball having the same volume as
$\Omega$. Bucur and Daners \cite{BD10} and, independently, Dai and Fu
\cite{DF11} obtained \eqref{ineq: BD} for the Robin $p$-Laplacian in Euclidean
space.
The Bossel-Daners inequality was established for complete noncompact Riemannian manifolds with $\Ric\geq 0$ in \cite{CLW23}, $n$-dimensional compact Riemannian manifolds with $\Ric\geq n-1$ and $n$-dimensional hyperbolic space in \cite{CCL22}.
Under a perimeter constraint,
Amato, Gentile, and Masiello \cite{AGM24} also obtained an
isoperimetric-deficit upper bound for the first eigenvalue of the Robin
$p$-Laplacian on convex subsets of $\mathbb{R}^n$. Related results can be found
in \cite{DK08,CG22,PP24} and the references therein.

The second part of this paper concerns positive boundary parameters. For smooth
bounded horospherically convex domains in $\mathbb{H}^n$, we establish the
following quantitative upper bound for the normalized eigenvalue difference
relative to the geodesic ball with the same perimeter.

\begin{thm}\label{thm4}
		Let $\Omega$ be a bounded horospherically convex domain with smooth boundary in hyperbolic space $\mathbb{H}^n$ ($n\geq 3$)
	and $P(\Omega)$ be the perimeter of $\Omega$. Let $\Omega^\star$ be the geodesic ball with the same perimeter as $\Omega$, that is, $P\ton{\Omega}=P\ton{\Omega^\star}$. For $\beta>0$, denote by $v$ the eigenfunction associated with $\lambda_1(\Omega^\star,\beta)$ for \eqref{eq:rob}. Then
	\begin{equation}\label{thm-posi}
		\frac{\lambda_1(\Omega,\beta)-\lambda_1(\Omega^\star,\beta)}
		{\lambda_1(\Omega,\beta)}
		\leq
		\frac{v_M^2}{\|v\|_{L^2(\Omega^\star)}^2}
		\bigl(|\Omega^\star|-|\Omega|\bigr),
	\end{equation}
	where
	\[
	v_M:=\max_{x\in\Omega^\star}v(x).
	\]
\end{thm}

\begin{remark}
	In the formal limit $\beta\to+\infty$, the Robin boundary condition reduces
	to the Dirichlet boundary condition. Accordingly, the inequality in Theorem
	\ref{thm4} may be viewed as an estimate for the Faber--Krahn deficit. In
	Euclidean space, Brandolini, Nitsch, and Trombetti \cite{BNT10} obtained a
	quantitative Faber--Krahn estimate in terms of the isoperimetric deficit.
\end{remark}

\begin{remark}
	An analogous estimate holds on the sphere. Let $\Omega\subset\mathbb{S}^n$ be
	an open set whose closure is strongly convex, and let $\Omega^\star$ be the
	strongly convex geodesic ball satisfying $P(\Omega^\star)=P(\Omega)$. For
	$\beta>0$, let $v$ be a first eigenfunction on $\Omega^\star$ corresponding to
	$\lambda_1(\Omega^\star,\beta)$, and set
	$v_M:=\max\limits_{x\in\Omega^\star}v(x)$. Then
	\begin{equation}\label{eq-thm3}
		\frac{\lambda_1(\Omega,\beta)-\lambda_1(\Omega^\star,\beta)}
		{\lambda_1(\Omega,\beta)}
		\leq
		\frac{v_M^2}{\|v\|_{L^2(\Omega^\star)}^2}
		\bigl(|\Omega^\star|-|\Omega|\bigr).
	\end{equation}
	Combining the relevant isoperimetric inequality \cite{Cha84}, the
	corresponding Bossel--Daners inequality \cite{CCL22,CLW23}, and the monotonicity
	of the first Robin eigenvalue with respect to the radius of concentric geodesic
	balls, we conclude that the left-hand sides of \eqref{thm-posi} and
	\eqref{eq-thm3} are nonnegative.
\end{remark}

The remainder of the paper is organized as follows. Section
\ref{Preliminaries} reviews the geometric notions needed below, together with
the Steiner formula and the Alexandrov--Fenchel inequality in hyperbolic space.
It also studies the convexity of parallel domains and derives an estimate for
the perimeter of inner parallel domains. The final section contains the proofs
of Theorems \ref{thm1} and \ref{thm4}.

	\section{Preliminaries}\label{Preliminaries}
This section collects the basic notions and geometric inequalities in
hyperbolic space that will be used below. We refer to
\cite{Koh91,GWW14,WX14} for further details. We also prove that
horospherical convexity is preserved under inner and outer parallel
deformations and derive a differential estimate for the perimeter of inner
parallel domains.

Let $\mathbb{R}^{n+1}_1$ denote the vector space $\mathbb{R}^{n+1}$ equipped
with the Lorentzian inner product
\begin{equation*}
	\ps{x}{y}
	=\ps{(x_0,x_1,\ldots,x_n)}{(y_0,y_1,\ldots,y_n)}
	=-x_0y_0+x_1y_1+\cdots+x_ny_n.
\end{equation*}
The hyperboloid model of the $n$-dimensional hyperbolic space of constant
sectional curvature $-1$ is
\begin{equation*}
	\mathbb{H}^n
	=\left\{x\in\mathbb{R}^{n+1}_1:\ps{x}{x}=-1,\ x_0\geq1\right\}.
\end{equation*}
	\subsection{Parallel hypersurfaces and  quermassintegrals}
	
	Let $\Omega\subset\mathbb{H}^n$ be a non-empty bounded domain, and denote by $R_\Omega$ its inradius, that is,
	\[
	R_\Omega=\max_{x\in\Omega}d\ton{x,\,\partial\Omega}.
	\]
	Given $0\le t<R_\Omega$ and $s>0$, the inner and outer parallel domains of $\Omega$ are defined by
	\[
	\Omega_t:=\bigl\{x\in\Omega:\ d(x,\partial\Omega)\ge t\bigr\},
	\qquad
	\Omega^s:=\bigl\{x\in\mathbb{H}^n:\ d(x,\Omega)\le s\bigr\},
	\]
where $d\ton{\cdot,\cdot}$ is the Riemannian distance, $d\ton{x,\,\partial \Omega}$ is the distance from $x\in \Omega$ to the boundary $\partial \Omega$ and $d(x, \Omega) = \inf\limits_{q \in \Omega} d(x, q)$.
	The Hausdorff distance between two compact sets $K_1,K_2\subset\mathbb{H}^n$ is defined by
	\[
	d(K_1,K_2)=\inf\Bigl\{t\ge0\ \Big|\ K_1\subset (K_2)^t,\ K_2\subset (K_1)^t\Bigr\}.
	\]

	We record the preservation of horospherical convexity under parallel
	deformations.
	\begin{prop}\label{h-convex}
		Let $\Omega$ be a bounded  horospherically convex domain with smooth boundary $\partial\Omega$ in $\mathbb{H}^n$. Then the inner parallel domain $\Omega_t$ is horospherically convex for every $0\le t<R_\Omega$, and the outer parallel domain $\Omega^s$ is horospherically convex for every $s>0$.
	\end{prop}
	
	\begin{proof}
		As $\mathbb{H}^n\subset\mathbb{R}_1^{n+1}$, let $N$ be the outward unit normal vector field to $\partial \Omega$, and denote by $\exp$ the exponential map of $\mathbb{H}^{n}$. By assumption, $\Omega$ is horospherically convex, so there exists a family $\mathcal{H}$ of horoballs such that
		\[
		\Omega = \bigcap_{H\in\mathcal{H}} H.
		\] 
		
		Let $q\in\Omega_t$ and $p\in\partial\Omega$ be arbitrary. Since
		$\Omega\subset H_p$, we have $d(q,\partial\Omega)\le d(q,\partial H_p)$.
		By definition of $\Omega_t$, $d(q,\partial\Omega)\ge t$, hence
		$d(q,\partial H_p)\ge t$. Thus $q\in(H_p)_t$, and because
		$(H_p)_t=\{z\in H_p:\ d(z,\partial H_p)\ge t\}$ is again a horoball,
		we obtain
		\begin{equation}\label{Omega_tsubsetH_p_t}
				\Omega_t\subset (H_p)_t\qquad\text{for every }p\in\partial\Omega.
		\end{equation}
		
		For any $q\in\partial\Omega_t$, there exists $p\in\partial\Omega$ such that
		$q=\exp_p(-tN(p))$. Since $\Omega\subset H_p$ and $p\in\partial\Omega\cap\partial H_p$,
		the hypersurfaces $\partial\Omega$ and $\partial H_p$ share the tangent plane and
		outward unit normal $N(p)$ at $p$. The inward normal geodesic rays of
		$\partial H_p$ are asymptotic to its ideal centre and orthogonal to every
		horosphere $\partial(H_p)_s$; hence
		\[
		d\bigl(\exp_p(-sN(p)),\,\partial H_p\bigr)=s\qquad 0\le s\le R_\Omega,
		\]
		and in particular $d(q,\partial H_p)=t$, that is, $q\in\partial(H_p)_t$.
		Together with \eqref{Omega_tsubsetH_p_t}, this shows that $(H_p)_t$ supports $\Omega_t$
		at $q$. Since every boundary point of $\Omega_t$ admits a supporting horoball of this
		form, $\Omega_t$ is horospherically convex.

		For the outer parallel domain $\Omega^s$, we adapt the argument of \cite{MR91}.
		Define
		\[
		x^s(p) = \exp_{x(p)}(sN(p)) = \cosh s\, x(p) + \sinh s\, N(p), \qquad p\in \partial \Omega.
		\]
		Let $\{\kappa_i\}_{i=1}^{n-1}$ be the principal curvatures of $\partial \Omega$ at $p$, with corresponding orthonormal principal directions $\{e_i\}_{i=1}^{n-1}$. Then,
		$$
		(x^s)_*(e_i) = \ton{\cosh s + \kappa_i \sinh s} e_i,
		$$
		and hence $N^s=\sinh s\,x+\cosh s\,N$ is a unit normal field of $x^s$. Moreover, the second
		fundamental form $\sigma^s$ of $x^s$ is given by 
		$$
		\sigma^s \left( (x^s)_*(e_i), (x^s)_*(e_j) \right) = \ps{ (N^s)_*(e_i)}{ (x^s)_*(e_j)}.
		$$
		By direct computation,
		$$
		\kappa_i(s) = \frac{\sinh s + \kappa_i \cosh s}{\cosh s +\kappa_i \sinh s} = \frac{\kappa_i +\tanh s}{1 + \kappa_i \tanh s}.
		$$
		Then 
		$$\kappa_i(s)-1=\frac{\ton{1-\tanh s}\ton{\kappa_i-1}}{1+\kappa_i\tanh s}\geq 0.$$
		Hence, $\Omega^s$ is a horospherically convex domain. 
	\end{proof}

	Let $\Omega\subset\mathbb{H}^n$ be a bounded horospherically convex domain
	with smooth boundary. For each $p\in\partial\Omega$, let $\kappa_1,\dots,\kappa_{n-1}$
	be the principal curvatures of $\partial\Omega$ at $p$.
	For $k=1,\ldots,n-1$, let $H_k$ denote
	the $k$th normalized elementary symmetric function of these principal
	curvatures, namely,
	\begin{equation*}
		H_k
		:=\binom{n-1}{k}^{-1}
		\sum_{1\leq i_1<\cdots<i_k\leq n-1}
		\kappa_{i_1}\cdots\kappa_{i_k},
		\qquad H_0:=1.
	\end{equation*}
	In particular, $H_1$ is the normalized mean curvature of $\partial \Omega$. For
	$s>0$, the Steiner formula for the perimeter of $\Omega^s$ is
	\begin{equation*}
		P(\Omega^s)
		=\sum_{i=0}^{n-1}\binom{n-1}{i}
		\cosh^i s\,\sinh^{n-1-i}s\,V_i(\Omega),
	\end{equation*}
	where the curvature integrals $V_i(\Omega)$ are defined by
	\begin{equation*}
		V_i(\Omega):=\int_{\partial \Omega}H_{n-1-i}\,d\sigma,
		\qquad i=0,\ldots,n-1.
	\end{equation*}
	Differentiating the Steiner formula at $s=0$ gives
	\begin{equation}\label{diff}
		\left.\frac{d}{ds}P(\Omega^s)\right|_{s=0}
		=(n-1)V_{n-2}(\Omega).
	\end{equation}
	
	The quermassintegrals of $\Omega$ are defined by
	\begin{equation*}
		W_k(\Omega)
		:=\frac{(n-k)\omega_{k-1}\cdots\omega_0}
		{n\omega_{n-2}\cdots\omega_{n-k-1}}
		\int_{\mathcal{L}_k}\chi(L_k\cap \Omega)\,dL_k,
		\qquad k=1,\ldots,n-1,
	\end{equation*}
	where $\omega_k$ is the $k$-dimensional Hausdorff measure of the unit
	$k$-sphere, $\mathcal{L}_k$ is the space of $k$-dimensional totally geodesic
	subspaces of $\mathbb{H}^n$, $dL_k$ is the natural invariant measure on
	$\mathcal{L}_k$. The function $\chi$ is defined to be $1$ if $L_k \cap \Omega \neq \emptyset$ and to be $0$ otherwise. In particular,
	\begin{equation*}
		W_0(\Omega)=|\Omega|,
		\qquad
		W_1(\Omega)=\frac{1}{n}P(\Omega),
		\qquad
		W_n(\Omega)=\frac{\omega_{n-1}}{n}.
	\end{equation*}
	By \cite[Proposition~7]{Sol05}, the quermassintegrals and curvature
	integrals satisfy
	\begin{equation*}
		V_{n-1-k}(\Omega)
		=n\left(W_{k+1}(\Omega)+\frac{k}{n-k+1}W_{k-1}(\Omega)\right),
		\qquad k=1,\ldots,n-1,
	\end{equation*}
	and
	\begin{equation*}
		V_{n-1}(\Omega)=nW_1(\Omega)=P(\Omega).
	\end{equation*}

	\subsection{The Alexandrov--Fenchel inequality}\label{subsec:parallel}
	To prove the hyperbolic Alexandrov--Fenchel inequality for closed bounded
	domains with non-smooth boundary, we first recall the Busemann function.
	For each ideal point $\xi\in\partial_\infty\mathbb{H}^n$, there is a smooth
	function $b_\xi\colon\mathbb{H}^n\to\mathbb{R}$ defined by
	\[
	b_\xi(x)=\lim_{t\to+\infty}\bigl(d(x,\gamma(t))-t\bigr),
	\]
	where $\gamma\colon[0,+\infty)\to\mathbb{H}^n$ is any unit-speed geodesic
	ray with $\gamma(+\infty)=\xi$, and the limit is well-defined up to an
	additive constant. Let $g$ be the hyperbolic metric on $\mathbb{H}^n$ and $\nabla$ its
	Levi-Civita connection. For a smooth function $u$ on $\mathbb{H}^n$,
	$du$ denotes its differential, so $du(X)=g(\nabla u,X)$. We write
	$|\nabla u|$ for the $g$-norm of its gradient and $\nabla^2u$ for its
	Hessian.

\begin{lemma}\label{busemann}
	For every ideal point $\xi\in\partial_\infty\mathbb{H}^n$, one may choose a Busemann function $b_\xi$ such that
	\begin{align}\label{bus}
		|\nabla b_\xi|=1,\qquad \nabla^2b_\xi= g-db_\xi\otimes db_\xi.
	\end{align}
	Its sublevel sets $\{b_\xi\le c\}$, $c\in\mathbb{R}$, are closed horoballs, and their boundary horospheres, equipped with the outward normal $\nabla b_\xi$, have all their principal curvatures equal to $1$.
\end{lemma}

\begin{proof}
Let $\ell$ be a future-directed null vector representing $\xi$ and set
	\[
	f(x):=-\langle x,\ell\rangle>0,\qquad b_\xi:=\log f.
	\]
	Then $\nabla f=-\ell+fx$. Since $\langle\ell,\ell\rangle=0$ and $\langle x,\ell\rangle=-f$, one has $|\nabla f|^2=f^2$. By direct computation, we have 
	$$D_X(\nabla f)=X(f)x+fX,\qquad X\in T\mathbb{H}^n,$$
 the first term is normal to $\mathbb{H}^n$, so $\nabla^2f=f g$. Consequently
	\[
	|\nabla b_\xi|^2=\frac{|\nabla f|^2}{f^2}=1,
	\qquad
	\nabla^2b_\xi=\frac{\nabla^2f}{f}-\frac{df\otimes df}{f^2}= g-db_\xi\otimes db_\xi.
	\]
	If $X$ is tangent to a level set of $b_\xi$, then $db_\xi(X)=0$ and hence
	\[
	\mathrm{II}(X,X)=\ps{\nabla_X\nabla b_\xi}{X}=\nabla^2b_\xi(X,X)=|X|^2,
	\]
	so every principal curvature of the level horosphere equals $1$.
\end{proof}

\begin{lemma}\label{finite}
	Let $K\subset\mathbb{H}^n$ be a compact horospherically convex body. For every $\delta>0$, there exist finitely many closed horoballs $H_1,\dots,H_N$ such that
	\begin{align}\label{Pdelta}
			K\subset P_\delta:=\bigcap_{i=1}^NH_i\subset K^\delta,
	\end{align}
	and $P_\delta$ is compact.
\end{lemma}

\begin{proof}
	Choose a family $\mathcal{H}$ of closed horoballs with
	\[
	K=\bigcap_{H\in\mathcal{H}}H.
	\]
	Let $\overline{\mathbb{H}^n}=\mathbb{H}^n\cup\partial_\infty\mathbb{H}^n$.
	The closure of a horoball in $\overline{\mathbb{H}^n}$ meets the ideal boundary only at its centre. 
	
We claim that no ideal point belongs to the closures of all members
of $\mathcal{H}$. The proof is by contradiction. If
$\xi\in\partial_\infty\mathbb{H}^n$ belonged to every such closure,
then every $H\in\mathcal{H}$ would have centre $\xi$.
Fix a normalization of $b_\xi$ and write $H=\{b_\xi\le c_H\}$. Since $K\subset H$ for every $H$, one has $c_H\ge\max_K b_\xi$, hence

$$K=\bigcap_{H\in\mathcal{H}}H=\bigl\{b_\xi\le c_0\bigr\},\qquad
c_0:=\inf_{H\in\mathcal{H}}c_H\ge\max_K b_\xi.$$
	The set on the right is an unbounded horoball, contradicting the
	compactness of $K$. Hence
	$$
	\bigl\{\partial_\infty\mathbb{H}^n\setminus\overline{H}:\ H\in\mathcal{H}\bigr\}
	$$
	is an open cover of the compact ideal sphere. Take a finite subcover consisting of horoballs $H_1^0,\dots,H_M^0$. Set
	$$
	P_0:=\bigcap_{i=1}^M H_i^0.
	$$
	Then $P_0$ is compact since its closure in $\overline{\mathbb{H}^n}$ avoids
	$\partial_\infty\mathbb{H}^n$.
	Define
	\[
	C_\delta:=\{x\in P_0:\operatorname{dist}(x,K)\ge\delta\},
	\]
	which is compact since it is a closed subset of $P_0$. 
	For each $x\in C_\delta$, the equality $K=\bigcap_{H\in\mathcal{H}}H$ yields a horoball $H_x\in\mathcal{H}$ with $x\notin H_x$. Thus the open sets $\mathbb{H}^n\setminus H_x$ cover $C_\delta$, and by compactness some finite subcover $\{\mathbb{H}^n\setminus H_{x_1},\dots,\mathbb{H}^n\setminus H_{x_m}\}$ exists.
	 Then
	$$
	P_\delta:=\bigcap_{i=1}^{M} H_i^0\cap\bigcap_{j=1}^{m} H_{x_j}
	$$
	is compact and satisfies $K\subset P_\delta\subset K^\delta$, which proves \eqref{Pdelta}.
\end{proof}

\begin{lemma}\label{h-convex approx}
	For every compact horospherically convex body $K\subset\mathbb{H}^n$, there exists a sequence of compact horospherically convex bodies $K_j$ with smooth boundary such that
	\[
	K\subset K_j,\qquad d(K_j,K)\to0. 
	\]
\end{lemma}

\begin{proof}
	Fix $\delta>0$ and let
	\[
	P=\bigcap_{i=1}^NH_i.
	\]
By Lemma~\ref{finite},
$$\qquad K\subset P\subset K^\delta.$$
 Write
	\[
	H_i=\{b_i\le c_i\},\qquad f_i:=b_i-c_i,\qquad G:=\max_{1\le i\le N}f_i,
	\]
	so that $P=\{G\le0\}$. For $0<\varepsilon\le1$, define
	\begin{equation}\label{Feps}
		F_\varepsilon:=\varepsilon\log\Bigl(\sum_{i=1}^Ne^{\frac{f_i}{\varepsilon}}\Bigr).
	\end{equation}
	The elementary log-sum-exp bounds give
	\begin{equation}\label{lsebound}
		G\le F_\varepsilon\le G+\varepsilon\log N.
	\end{equation}
	By Sard's theorem, choose a regular value

	\[
	a_\varepsilon\in\bigl(\varepsilon\log N+\varepsilon^2,\ \varepsilon\log N+2\varepsilon^2\bigr).
	\]
	Set $K_{\delta,\varepsilon}:=\{F_\varepsilon\le a_\varepsilon\}$. For $x\in P$, \eqref{lsebound} gives $F_\varepsilon(x)\le\varepsilon\log N<a_\varepsilon$, so
	\begin{equation}\label{incl1}
		P\subset K_{\delta,\varepsilon}.
	\end{equation}
	Conversely, $G\le F_\varepsilon$ gives
	\begin{equation}\label{incl2}
		K_{\delta,\varepsilon}\subset Q_{a_\varepsilon}:=\{G\le a_\varepsilon\}=\bigcap_{i=1}^N\{f_i\le a_\varepsilon\}.
	\end{equation}
	The sets $Q_a$ decrease to $P$ as $a\to0^+$ and are compact for all sufficiently small $a$. Since shifting the level set of a horoball preserves its ideal centre, the closure of $Q_a$ in $\overline{\mathbb{H}^n}$ lies in that of the defining horoballs' intersection and hence contains no ideal point. We claim that
	\begin{equation}\label{Qaconv}
		d(Q_a,P)\to0\qquad \text{as}~a\to0^+.
	\end{equation}
	Otherwise, one could choose $a_m\to0^+$ and points $x_m\in Q_{a_m}$ at a fixed positive distance from $P$. Since $Q_{a_m}\subset Q_{a_1}$ for large $m$ and $Q_{a_1}$ is compact, a subsequence converges to some $x$. Then $f_i(x)\le0$ for every $i$, so, ~$x\in P$, a contradiction. It follows from \eqref{incl1}--\eqref{Qaconv} that $K_{\delta,\varepsilon}\to P$ in the Hausdorff distance as $\varepsilon\to0^+$. Since $d(P,K)\le\delta$, by Lemma~\ref{finite}, choosing $\varepsilon$ sufficiently small gives
	\[
	d(K_{\delta,\varepsilon},K)\le d(K_{\delta,\varepsilon},P)+d(P,K)\le 2\delta.
	\]
	As $a_\varepsilon$ is a regular value, the boundary $\partial K_{\delta,\varepsilon}=\{F_\varepsilon=a_\varepsilon\}$ is smooth. By \eqref{incl2} and the compactness of $Q_{a_\varepsilon}$, the body $K_{\delta,\varepsilon}$ is compact.
	
	It remains to verify horospherical convexity. Put
	\[
	S_\varepsilon:=\sum_{i=1}^Ne^{\frac{f_i}{\varepsilon}}.
	\]
	Then $F_\varepsilon=\varepsilon\log S_\varepsilon$, so  $F_\varepsilon$ and $S_\varepsilon$ have the same level sets. By \eqref{bus}, $\nabla^2f_i= g-df_i\otimes df_i$, and therefore
	\begin{equation}\label{convS}
		\begin{split}
			\nabla^2\ton{e^{\frac{f_i}{\varepsilon}}}=\frac{e^{\frac{f_i}{\varepsilon}}}{\varepsilon} g+\frac{e^{\frac{f_i}{\varepsilon}}}{\varepsilon}\Bigl(\frac1\varepsilon-1\Bigr)df_i\otimes df_i
			\succeq\frac{e^{\frac{f_i}{\varepsilon}}}{\varepsilon} g,
		\end{split}
	\end{equation}
	where the last inequality uses $\varepsilon\le1$. Thus $\nabla^2S_\varepsilon\succeq\frac{S_\varepsilon}{\varepsilon} g\succ0$, so $S_\varepsilon$ is strictly geodesically convex and $K_{\delta,\varepsilon}=\{S_\varepsilon\le e^{\frac{a_\varepsilon}{\varepsilon}}\}$ is geodesically convex.
	
To estimate the principal curvatures of $\partial K_{\delta,\varepsilon}$, introduce the weights
	\[
	\theta_i:=\frac{e^{\frac{f_i}{\varepsilon}}}{S_\varepsilon},\qquad \theta_i>0,\qquad \sum_i\theta_i=1.
	\]
	Differentiating \eqref{Feps} gives
	\begin{align}\label{dF}
		dF_\varepsilon=\sum_i\theta_i\,df_i,
	\end{align}
	and
	\begin{align}\label{HF}
		\nabla^2F_\varepsilon=\sum_i\theta_i\nabla^2f_i+\frac1\varepsilon\Bigl(\sum_i\theta_i\,df_i\otimes df_i-dF_\varepsilon\otimes dF_\varepsilon\Bigr).
	\end{align}
	Let $X$ be a tangent vector t0 $\partial K_{\delta,\varepsilon}$, so $dF_\varepsilon(X)=0$. Using \eqref{bus} and \eqref{HF},
	\begin{align}\label{key}
		\nabla^2F_\varepsilon(X,X)=|X|^2+\Bigl(\frac1\varepsilon-1\Bigr)\sum_i\theta_i\bigl(df_i(X)\bigr)^2\ge|X|^2.
	\end{align}
	Moreover, by \eqref{dF} and $|\nabla f_i|=1$,
	\[
	|\nabla F_\varepsilon|\le\sum_i\theta_i|\nabla f_i|=1. 
	\]
	The outward unit normal vector  to $\partial K_{\delta,\varepsilon}$ is $\nu=\frac{\nabla F_\varepsilon}{|\nabla F_\varepsilon|}$. Hence the second fundamental form satisfies
	\[
	\mathrm{II}(X,X)=\frac{\nabla^2F_\varepsilon(X,X)}{|\nabla F_\varepsilon|}\ge|X|^2.
	\]
	Thus every principal curvature of $\partial K_{\delta,\varepsilon}$ is at least $1$, and hence $K_{\delta,\varepsilon}$ is horospherically convex.
	
	Finally, take $\delta_j\to0^+$. For each $j$, Lemma~\ref{finite} gives a finite intersection $P_j$ with $K\subset P_j\subset K^{\delta_j}$; choose $\varepsilon_j$ so small that $d(K_{\delta_j,\varepsilon_j},P_j)<\delta_j$. Then
	\[
	K_j:=K_{\delta_j,\varepsilon_j}
	\quad\text{satisfies}\quad
	K\subset K_j\subset K^{2\delta_j},
	\]
	and each $K_j$ is a compact horospherically convex body with smooth boundary by the preceding regularity discussion, which proves Lemma~\ref{h-convex approx}.
\end{proof}

Recall the following hyperbolic Alexandrov--Fenchel inequality for smooth horospherically convex domains. The case $k=1$ and the case of $2k\le n-1$ were proved in \cite{GWW14}, while the case $1\le k\le n-1$ was proved in \cite{WX14}.
\begin{thm}\label{V_k}
	Let $1\le k\le n-1$. Any horospherically convex bounded domain $\Omega$ with smooth boundary in $\mathbb{H}^n$ satisfies
	\[
	\int_{\partial\Omega}H_k\,d\sigma
	\ge
	\omega_{n-1}\Biggl\{\Bigl(\frac{P(\Omega)}{\omega_{n-1}}\Bigr)^{\frac2k}
	+\Bigl(\frac{P(\Omega)}{\omega_{n-1}}\Bigr)^{\frac2k\frac{(n-k-1)}{n-1}}\Biggr\}^{\frac{k}{2}}.
	\]
	Equality holds if and only if $\Omega$ is a geodesic ball.
\end{thm}
We now establish Theorem~\ref{V_k} for horospherically convex domains with possibly non-smooth boundary.
\begin{thm}\label{Alexandrov-Fenchel}
	Let $1\le k\le n-1$. Any closed horospherically convex bounded domain $C$ in $\mathbb{H}^n$ satisfies
	\begin{equation}\label{eq_AF}
		\int_{\partial C}H_k\,d\sigma
		\ge
		\omega_{n-1}\Biggl\{\Bigl(\frac{P(C)}{\omega_{n-1}}\Bigr)^{\frac2k}
		+\Bigl(\frac{P(C)}{\omega_{n-1}}\Bigr)^{\frac2k\frac{(n-k-1)}{n-1}}\Biggr\}^{\frac{k}{2}}.
	\end{equation}
	Equality holds if $C$ is a geodesic ball.
\end{thm}
\begin{proof}
	By Lemma~\ref{h-convex approx}, there exists a sequence of compact horospherically convex sets $C_l$ with smooth boundary such that, as $l\to\infty$,
	\[
	C\subset C_l,\qquad d(C_l,C)\to0.
	\]
	By Theorem~\ref{V_k}, for every $l$,
	\[
	\int_{\partial C_l}H_k\,d\sigma
	\ge
	\omega_{n-1}\Biggl\{\Bigl(\frac{P(C_l)}{\omega_{n-1}}\Bigr)^{\frac2k}
	+\Bigl(\frac{P(C_l)}{\omega_{n-1}}\Bigr)^{\frac2k\frac{(n-k-1)}{n-1}}\Biggr\}^{\frac{k}{2}}.
	\]
   It follows from \cite[Theorem 2.4 and Lemma 2.2]{Koh94} (see also \cite{ACC+25}) that
	\[
	V_i(C_l;\cdot)\to V_i(C;\cdot),\qquad P(C_l)\to P(C),\qquad i=0,\dots,n-1,
	\]
	as $l\to\infty$, which implies \eqref{eq_AF}.
\end{proof}

By combining Proposition~\ref{h-convex} with Theorem~\ref{Alexandrov-Fenchel}, we obtain an estimate for the derivative of the perimeter of $\Omega_t$.

\begin{lemma}\label{diff-P}
	Let $\Omega$ be a bounded horospherically convex  domain with smooth boundary in $\mathbb{H}^n$. Then for almost every $t\in(0,R_\Omega)$,
	\[
	-\frac{d}{dt}P\ton{\Omega_t}
	\ge
	(n-1)\omega_{n-1}\Biggl\{\Bigl(\frac{P(\Omega_t)}{\omega_{n-1}}\Bigr)^{2}
	+\Bigl(\frac{P(\Omega_t)}{\omega_{n-1}}\Bigr)^{\frac{2(n-2)}{n-1}}\Biggr\}^{\frac12}.
	\]
	Equality holds if $\Omega$ is a geodesic ball.
\end{lemma}

\begin{proof}
	By Proposition \ref{h-convex}, the set $\Omega_t$ is a horospherically convex domain, which implies that $\ton{\Omega_t}^s$ is also a horospherically convex domain for $s>0$. Notice that 
$$\ton{\Omega_t}^s\subset \Omega_{t-s}.$$
The function $P\ton{\Omega_t}$ is nonincreasing with respect to $t$, hence it is differentiable almost everywhere. For almost every $t\in\ton{0,R_\Omega}$, due to \eqref{diff}, one can infer that
\begin{align*}
	-\frac{d}{dt}P\ton{\Omega_t}&= \lim_{s \to 0^+} \frac{P(\Omega_{t - s}) - P(\Omega_t)}{s}\\
	&\geq \lim_{s \to 0^+} \frac{P((\Omega_t)^s) - P(\Omega_t)}{s}\\
	&=(n-1)V_{n-2}\ton{\Omega_t}.
\end{align*}
	Finally, Theorem~\ref{Alexandrov-Fenchel} with $k=1$ applied to the horospherically convex domain $\Omega_t$ gives
	\[
	V_{n-2}(\Omega_t)=\int_{\partial\Omega_t}H_1\,d\sigma
	\ge
	\omega_{n-1}\Biggl\{\Bigl(\frac{P(\Omega_t)}{\omega_{n-1}}\Bigr)^{2}
	+\Bigl(\frac{P(\Omega_t)}{\omega_{n-1}}\Bigr)^{\frac{2(n-2)}{n-1}}\Biggr\}^{\frac12},
	\]
	which completes the proof.
\end{proof}

	\section{Proofs of Theorems \ref{thm1} and \ref{thm4}}
	In this section, inspired by \cite{Pol60, PW61, ACC+25}, we prove Theorems \ref{thm1} and \ref{thm4} using the method of parallel coordinates to construct a suitable test function.

	Without loss of generality, we assume that $\Omega^\star$ is the geodesic ball centered at $q\in\mathbb{H}^n$ with radius $R>0$ satisfying $P\ton{\Omega^\star}=P\ton{\Omega}$. The eigenfunction $v(x)=\psi\ton{d\ton{x,q}}>0$ associated with $\lambda_1\ton{\Omega^\star,\beta}$ is the radial solution to the following equation
	\begin{eqnarray}\label{rad-rob}
		\begin{cases}
			\psi'' + (n-1)\dfrac{\cosh r}{\sinh r} \psi' + \lambda_1\ton{\Omega^\star,\beta} \psi = 0 & r \in (0, R), \\
			\psi'(0) = 0, \\
			\psi'(R) + \beta \psi(R) = 0.
		\end{cases}
	\end{eqnarray}
	Letting $\phi\ton{\rho} = \psi\ton{R - \rho}$, we can regard $v$ as a function of the distance to the boundary of the ball $\Omega^\star$, that is,
	$$v(x) = \phi\ton{d\ton{x, \partial\Omega^\star}}.$$
	Choosing $u=1$ in \eqref{eq:ray}, it can be concluded that $\lambda_1\ton{\Omega^\star,\beta}<0$ for $\beta<0$. Moreover, by \eqref{rad-rob},
	\begin{align*}
		\dfrac{\partial}{\partial r}\ton{\sinh ^{n-1}r\,v^\prime(r)}=-\lambda_1\ton{\Omega^\star,\beta}v\ton{r}\sinh^{n-1}r>0,
	\end{align*}
	which implies that $v^\prime(r)>0$ for $r\in(0,R]$, that is, $v\ton{r}$ is strictly increasing with respect to $r$. Similarly, if $\beta>0$, then $v^\prime(r)<0$ for $r\in(0,R]$, that is, $v\ton{r}$ is strictly decreasing with respect to $r$.
	
	\begin{proof}[Proof of Theorem \ref{thm1}]
		Let $v>0$ be the eigenfunction associated with the eigenvalue  $\lambda_1\left(\Omega^\star,\beta\right)$ for \eqref{rad-rob}.
		Since $v'(r)>0$ for $r\in(0,R]$, where $R$ is the radius of $\Omega^\star$, one has $\phi'(t)<0$ for $t\in[0,R)$. Define
		$$u(x)=\phi\ton{d\ton{x,\,\partial\Omega}},\qquad x\in\Omega.$$
		Then
		\begin{align}\label{eq:min}
			v_m=\min\limits_{x\in\Omega^\star}v(x)=\phi\ton{R}\leq\phi\ton{R_\Omega}=\min\limits_{x\in\Omega}u(x)=u_m.
		\end{align}		 
		By Lemma \ref{diff-P}, we have
		\begin{align*}
			\frac{d}{dt}P\ton{\Omega_t}\leq-\ton{n-1}\omega_{n - 1} \left\{ \left( \frac{P\ton{ \Omega_t}}{\omega_{n - 1}} \right)^{2} + \left( \frac{P\ton{\Omega_t}}{\omega_{n - 1}} \right)^{ \frac{2(n -2)}{n-1}} \right\}^{\frac{1}{2}}
		\end{align*}
		and 
		\begin{align*}
			\frac{d}{dt}P\ton{\Omega^\star_t}=-\ton{n-1}\omega_{n - 1} \left\{ \left( \frac{P\ton{\Omega^\star_t}}{\omega_{n - 1}} \right)^{2} + \left( \frac{P\ton{\Omega^\star_t}}{\omega_{n - 1}} \right)^{ \frac{2(n -2)}{n-1}} \right\}^{\frac{1}{2}}.
		\end{align*}
		From Lemma 3.3 in \cite{ACC+25}, it follows that
		\begin{align}\label{ineq:P}
			P\ton{\Omega_t}\leq P\ton{\Omega^\star_t},\qquad t\in\ton{0,R_\Omega}.
		\end{align}		 
		Combining \eqref{eq:min}, \eqref{ineq:P} with the co-area formula yields
		\begin{equation*}
			\begin{split}
				\int_{\Omega} u^2 \, d\mu &= \int_{0}^{R_{\Omega}} \phi^2(t) P\ton{\Omega_t} \, dt \\
				&= \int_{0}^{R} \phi^2(t) P\ton{\Omega^\star_t} \, dt-\int_{R_\Omega}^{R} \phi^2(t) P\ton{\Omega^\star_t} \, dt -\int_{0}^{R_\Omega} \phi^2(t) \ton{P\ton{\Omega^\star_t}-P\ton{\Omega_t}} \, dt \\
				&\leq\int_{0}^{R} \phi^2(t) P\ton{\Omega^\star_t}dt-\phi^2(R)\ton{\int_0^RP(\Omega^\star_t) dt- \int_0^{R_\Omega}P(\Omega_t) dt}\\
				&= \int_{\Omega^\star} v^2 \, d\mu -v_m^2\ton{\abs{\Omega^\star}-\abs{\Omega}},
			\end{split}
		\end{equation*}		
		and
		\begin{eqnarray*}
			\int_{\Omega} |\nabla u|^2 \, d\mu = \int_{0}^{R_{\Omega}} (\phi'(t))^2 P(\Omega_t) \, dt \leq \int_{0}^{R} (\phi'(t))^2 P(\Omega^\star_t) \, dt=\int_{\Omega^\star} |\nabla v|^2 \, d\mu.
		\end{eqnarray*}
		A direct computation shows that
		\begin{eqnarray*}
			\int_{\partial\Omega}u^2\,d\sigma=\phi^2(0)P\ton{\Omega}=\int_{\partial\Omega^\star}v^2\,d\sigma.
		\end{eqnarray*}
		Hence, we have
		\begin{align*}
			\lambda_1\ton{\Omega,\beta}&\leq \dfrac{\displaystyle\int_{\Omega} |\nabla u|^2\, d\mu + \beta \int_{\partial \Omega} u^2\, d\sigma}{\displaystyle\int_{\Omega} u^2\, d\mu}\\
			&\leq\dfrac{\displaystyle\int_{\Omega^\star} |\nabla v|^2\, d\mu + \beta \int_{\partial \Omega^\star} v^2\, d\sigma}{\displaystyle\int_{\Omega^\star} v^2 \, d\mu -v_m^2\ton{\abs{\Omega^\star}-\abs{\Omega}}}\\
			&=\dfrac{\displaystyle\lambda_1\ton{\Omega^\star,\beta}}{\displaystyle 1-\frac{\displaystyle v^2_m\ton{\abs{\Omega^\star}-\abs{\Omega}}}{\displaystyle\|v\|_{L^2\ton{\Omega^\star}}^2}}.
		\end{align*}		
		This completes the proof of Theorem \ref{thm1}.
	\end{proof}

	\begin{proof}[Proof of Theorem \ref{thm4}]
		For $\beta>0$, let $v>0$ be the eigenfunction associated with $\lambda_1\left(\Omega^\star,\beta\right)$ for the problem \eqref{rad-rob}. Then $\phi'(t)>0$ for $t\in[0,R)$. 
		Define
		$$u(x)=\phi\ton{d\ton{x,\,\partial\Omega}}.$$
		Then
		\begin{align}\label{eq:v}
			u_M=\max\limits_{x\in\Omega}u(x)=\phi\ton{R_\Omega}\leq \phi\ton{R}=\max\limits_{x\in\Omega^\star}v(x)=v_M.
		\end{align}
		The definitions of $u$ and $v$ yield
		\begin{align}\label{eq:Re2}
			\int_{\partial\Omega} u^2\, d\sigma=\phi^2(0)P\ton{\Omega}=\phi^2(0)P\ton{\Omega^\star}=\int_{{\partial\Omega}^\star}v^2 \,d\sigma.
		\end{align}
		From \eqref{ineq:P}, \eqref{eq:v} and the co-area formula, it follows that
		\begin{equation}\label{eq:Re1}
			\begin{split}
				\int_\Omega\abs{\nabla u}^2\,d\mu &=\int_0^{R_\Omega}\ton{\phi'(t)}^2P\ton{\Omega_t}\,dt\\
				&\leq\int_0^{R}\ton{\phi'(t)}^2P\ton{\Omega^\star_t}\,dt\\
				&=\int_{\Omega^\star}\abs{\nabla v}^2\,d\mu,
			\end{split}
		\end{equation}
		and 
		\begin{equation}\label{eq:Re3}
			\begin{split}
				\int_\Omega u^2\,d\mu=&\int_0^{R_\Omega}\phi^2(t)P\ton{\Omega_t}\,dt\\
			=& \int_{0}^{R} \phi^2(t) P\ton{\Omega^\star_t} \, dt-\int_{R_\Omega}^{R} \phi^2(t) P\ton{\Omega^\star_t} \, dt -\int_{0}^{R_\Omega} \phi^2(t) \ton{P\ton{\Omega^\star_t}-P\ton{\Omega_t}} \, dt \\
			\geq&\int_{0}^{R} \phi^2(t) P\ton{\Omega^\star_t}dt-\phi^2(R)\ton{\int_0^RP(\Omega^\star_t) dt- \int_0^{R_\Omega}P(\Omega_t) dt}\\
				=&\int_{\Omega^\star} v^2\,d\mu-v_M^2\ton{\abs{\Omega^\star}-\abs{\Omega}}.
			\end{split}
		\end{equation}	
		If the inequality
		\[
		\int_{\Omega^\star} v^2\,d\mu - v_M^2 \left(|\Omega^\star| - |\Omega|\right) \leq
		0
		\]
		holds, it follows that
		\[
		\frac{\displaystyle\lambda_1(\Omega,\beta) - \lambda_1(\Omega^\star,\beta)}{\displaystyle\lambda_1(\Omega,\beta)} \leq 1 \leq \frac{\displaystyle v_M^2 \left(|\Omega^\star| - |\Omega|\right)}{\displaystyle \|v\|_{L^2(\Omega^\star
				)}^2}.
		\]
		Otherwise, if
		\[
		\int_{\Omega^\star} v^2\,d\mu - v_M^2 \left(|\Omega^\star| - |\Omega|\right
		) > 0,
		\]
		then combining 
		\eqref{eq:Re2}, \eqref{eq:Re1} and \eqref{eq:Re3}
		yields
		\begin{align*}
			\lambda_1\ton{\Omega,\beta
			}
			&\leq \frac{\displaystyle\int_{\Omega} |\nabla u|^2\, d\mu + \beta \int_{\partial \Omega} u^2\, d\sigma}{\displaystyle\int_{\Omega} u^2\, d\mu} \\
			&\leq \frac{\displaystyle\int_{\Omega^\star} |\nabla v|^2\, d\mu + \beta \int_{\partial \Omega^\star} v^2\, d\sigma}{\displaystyle\int_{\Omega^\star} v^2 \, d\mu - v_M^2 \left(|\Omega^\star| - |\Omega|\right)} \\
			&= 
			\frac{\lambda_1(\Omega^\star,\beta)}{\displaystyle 1 - \frac{v_M^2 \left(|\Omega^\star| - |\Omega|\right)}{\|v\|_{L^2(\Omega^\star
						)}^2}}.
		\end{align*}
		This completes the proof of Theorem~
		\ref{thm4}.
	\end{proof}
	
	\subsection*{Acknowledgements} The authors thank Professor Yong Wei for his helpful discussions.

\vspace{1cm}

\begin{flushleft}
	Daguang Chen: Department of Mathematical Sciences, Tsinghua University, Beijing, 100084, P.R. China;
	E-mail: dgchen@tsinghua.edu.cn\\
	Shan Li: CCTEG Wuhan Engineering Co., Ltd., Wuhan, 430064, P.R. China;
	E-mail: lishan22@tsinghua.org.cn\\

\end{flushleft}

\end{document}